\RequirePackage{amsmath}
\RequirePackage{amsthm}

\documentclass[envcountsect,smallextended]{svjour3}       
\smartqed  
\journalname{myjournal}
\usepackage{fix-cm}
\usepackage{enumerate}
\usepackage{amssymb}
\usepackage{cite}
\numberwithin{equation}{section}
\newtheorem{innercustomthm}{Theorem}
\newenvironment{customthm}[1]
  {\renewcommand\theinnercustomthm{#1}\innercustomthm}
  {\endinnercustomthm}

\spnewtheorem{thm}{Theorem}[section]{\bfseries}{\itshape}
\spnewtheorem{cor}[thm]{Corollary}{\bfseries}{\itshape}
\spnewtheorem{lem}[thm]{Lemma}{\bfseries}{\itshape}
\spnewtheorem{ntn}{Notations}[section]{\bfseries}{\itshape}
\spnewtheorem{pro}{Proposition}[section]{\bfseries}{\itshape}
\spnewtheorem{dfn}{Definition}[section]{\bfseries}{\itshape}
\spnewtheorem{as}{Assumption}[section]{\bfseries}{\itshape}
\spnewtheorem{rem}{Remark}[section]{\bfseries}{\itshape}
\spnewtheorem{ob}{Observation}[section]{\bfseries}{\itshape}
\begin{document}

\title{The Dual Minkowski Problem for Negative Indices
}


\author{Yiming Zhao  
}


\institute{Yiming Zhao \at
              Department of Mathematics, Tandon School of Engineering, New York University, NY, USA\\
              \email{yiming.zhao@nyu.edu}           %
}

\date{Received: date / Accepted: date}

\maketitle

\begin{abstract}
Recently, the duals of Federer's curvature measures, called dual curvature measures, were discovered by Huang, Lutwak, Yang \& Zhang \cite{HLYZ}. In the same paper, they posed the dual Minkowski problem, the characterization problem for dual curvature measures, and proved existence results when the index, $q$, is in $(0,n)$. The dual Minkowski problem includes the Aleksandrov problem ($q=0$) and the logarithmic Minkowski problem ($q=n$) as special cases. In the current work, a complete solution to the dual Minkowski problem whenever $q<0$, including both existence and uniqueness, is presented.
\end{abstract}

\section{Introduction}
\label{section introduction}
Quermassintegrals, which include volume and surface area as special cases, are among the most fundamental geometric invariants in the Brunn-Minkowski theory of convex bodies (compact, convex subsets of $\mathbb{R}^n$ with non-empty interiors). Denote by $\mathcal{K}_o^n$ the set of all convex bodies with the origin in their interiors. For $i=1,\cdots, n$, the $(n-i)$-th quermassintegral $W_{n-i}(K)$ of $K\in \mathcal{K}_o^n$ can be seen as the mean of the projection areas of $K$:
\begin{equation}
\label{eq intro quermassintegral}
W_{n-i}(K)=\frac{\omega_n}{\omega_i}\int_{G(n,i)}\mathcal{H}^i(K|\xi)d\xi,
\end{equation}
where the integration is with respect to the Haar measure on the Grassmannian $G(n,i)$ containing all $i$ dimensional subspaces of $\mathbb{R}^n$. Here $K|\xi$ is the image of the orthogonal projection of $K$ onto $\xi \in G(n,i)$, $\mathcal{H}^i$ is the $i$ dimensional Hausdorff measure and $\omega_j$ is the $j$ dimensional volume of the unit ball in $\mathbb{R}^j$ for each positive integer $j$. Three families of geometric measures can be viewed as differentials of quermassintegrals: area measures, (Federer's) curvature measures, and $L_p$ surface area measures. If one replaces the orthogonal projection in \eqref{eq intro quermassintegral} by intersection, the fundamental geometric functionals in the dual Brunn-Minkowski theory will appear. The $(n-i)$-th dual quermassintegral $\widetilde{W}_{n-i}(K)$ can be defined by
\begin{equation}
\label{eq intro dual quermassintegral}
\widetilde{W}_{n-i}(K)=\frac{\omega_n}{\omega_i}\int_{G(n,i)}\mathcal{H}^i(K\cap \xi)d\xi.
\end{equation}
Compare \eqref{eq intro quermassintegral} with \eqref{eq intro dual quermassintegral}. As opposed to quermassintegrals which capture boundary information about convex bodies, dual quermassintegrals, the fundamental geometric functionals in the dual Brunn-Minkowski theory, encode interior properties of convex bodies. Arising from dual quermassintegrals is a new family of geometric measures, dual curvature measures $\widetilde{C}_q(K,\cdot)$ for $q\in \mathbb{R}$, discovered by Huang, Lutwak, Yang \& Zhang (Huang-LYZ) in their groundbreaking work \cite{HLYZ}. This new family of measures miraculously connects well-known measures like Aleksandrov's integral curvature and the cone volume measure. 

Huang-LYZ \cite{HLYZ} asked for the conditions on a given Borel measure $\mu$ necessary and sufficient for it to be exactly the $q$-th dual curvature measure of a convex body. This problem is called \emph{the dual Minkowski problem}. The Aleksandrov problem and the logarithmic Minkowski problem are important special cases of the dual Minkowski problem. In the case when $\mu$ has $f$ as its density, solving the dual Minkowski problem is equivalent to solving the following Monge-Amp\`{e}re type equation on the unit sphere $S^{n-1}$:
\begin{equation*}
\frac{1}{n}h(v)|\nabla_{S^{n-1}}h(v)+h(v)v|^{q-n}\text{det}(h_{ij}(v)+h(v)\delta_{ij})=f(v),
\end{equation*} 
where $\nabla_{S^{n-1}}h$ is the gradient of $h$ on $S^{n-1}$, $h_{ij}$ is the Hessian of $h$ with respect to an orthonormal frame on $S^{n-1}$, and $\delta_{ij}$ is the Kronecker delta. Huang-LYZ \cite{HLYZ} considered the dual Minkowski problem for $q\in (0,n)$ when the given measure is even and gave a sufficient condition that ensures the existence of a solution. The uniqueness of the solution remains open (except when $q=0$). Very recently, for the critical integer cases $q=1,2,\cdots, n-1$, a better sufficient condition for the existence part of the dual Minkowski problem when the given measure is even was given by the author in \cite{YZ}. Independently and simultaneously, the same condition was presented by B\"{o}r\"{o}czky, Henk \& Pollehn in \cite{BH} and shown by them to be necessary. In the current work, we consider the dual Minkowski problem for the case when $q<0$. A complete solution in this case, including existence and uniqueness, will be presented. 

The family, $L_p$ surface area measures $S_p(K,\cdot)$ for $p\in \mathbb{R}$, introduced by Lutwak \cite{MR1231704,MR1378681}, is the family of fundamental geometric measures in the $L_p$ Brunn-Minkowski theory and has appeared in a growing number of works, see, for example, Haberl \cite{MR2966660}, Haberl \& Parapatits \cite{MR3194492,MR3176613}, Haberl \& Schuster \cite{MR2545028}, and LYZ \cite{MR1863023,MR2142136}. When $p=0$, $L_p$ surface area measure is best known as the cone volume measure and is being intensively studied, see, for example, Barthe, Gu\'{e}don, Mendelson \& Naor \cite{MR2123199}, B\"{o}r\"{o}czky-LYZ \cite{MR2964630,BLYZ,MR3316972}, B\"{o}r\"{o}czky \& Henk \cite{MR3415694}, Henk \& Linke \cite{MR3148545}, Ludwig \& Reitzner \cite{MR2680490}, Stancu \cite{MR1901250,MR2019226}, Zhu \cite{MR3228445}, and Zou \& Xiong \cite{MR3255458}. 

The characterization problem for $L_p$ surface area measure is called the $L_p$ Minkowski problem, which asks for necessary and sufficient condition(s) on the given measure so that it is exactly the $L_p$ surface area measure of a convex body. The solution to the $L_p$ Minkowski problem when $p>1$ was given by Chou \& Wang \cite{MR2254308}. See also Chen \cite{MR2204749}, Lutwak \cite{MR1231704}, LYZ \cite{MR2067123}, and Zhu \cite{MR3352764}. The solution to the $L_p$ Minkowski problem has proven to be essential in establishing different analytic affine isoperimetric inequalities (see, for example, Haberl \& Schuster \cite{MR2530600}, LYZ \cite{MR1987375}, Wang \cite{MR2927377}, and Zhang \cite{MR1776095}). Two major unsolved cases of the $L_p$ Minkowski problem are: when $p=-n$, the centro-affine Minkowski problem (see Zhu \cite{MR3356071}) and when $p=0$, the logarithmic Minkowski problem. For the logarithmic Minkowski problem, a necessary and sufficient condition has been given by B\"{o}r\"{o}czky-LYZ \cite{MR2964630} in the even case to ensure the existence of a solution. When the measure is arbitrary, different efforts have been made by B\"{o}r\"{o}czky, Heged\H{u}s \& Zhu \cite{Boroczky20062015}, Stancu \cite{MR1901250,MR2019226}, and Zhu \cite{MR3228445}. In both cases, the uniqueness part of the logarithmic Minkowski problem has proven to be extremely difficult and still remains open. See B{\"o}r{\"o}czky-LYZ \cite{MR2964630} and Stancu \cite{MR1901250,MR2019226} for some progress in the planar case. The logarithmic Minkowski problem has strong connections with whether a measure has an affine isotropic image (B{\"o}r{\"o}czky-LYZ \cite{MR3316972}) and curvature flows (Andrews \cite{MR1714339,MR1949167}).

The first curvature measure $C_0(K,\cdot)$ is also known as Aleksandrov's integral curvature. Its characterization problem is called the Aleksandrov problem whose solution was given by Aleksandrov using a topological argument, namely, his Mapping Lemma \cite{MR0007625}. A very interesting new approach to solving the Aleksandrov problem was presented by Oliker \cite{MR2332603}. Aleksandrov's integral curvature (the Aleksandrov problem resp.) and the cone volume measure (the logarithmic Minkowski problem resp.) were never thought to be connected until the recent remarkable work \cite{HLYZ} by Huang-LYZ. They discovered a family of geometric measures, dual curvature measures $\widetilde{C}_q(K,\cdot)$ for $q\in \mathbb{R}$, that can be viewed as differentials of dual quermassintegrals. Readers are recommended to see Section 3 in \cite{HLYZ} for the detailed construction of dual curvature measures. The family of dual curvature measures serves as a bridge linking Aleksandrov's integral curvature measure and the cone volume measure. When $q=0$, $0$-th dual curvature measure is (up to a constant) the same as Aleksandrov's integral curvature for the polar body. When $q=n$, $n$-th dual curvature measure is (up to a constant) equal to the cone volume measure. Huang-LYZ posed the characterization problem for dual curvature measures:

\textbf{The dual Minkowski problem} \cite{HLYZ}:\textit{
Given a finite Borel measure $\mu$ on $S^{n-1}$ and $q\in \mathbb{R}$, find the necessary and sufficient condition(s) on $\mu$ so that there exists a convex body $K$ containing the origin in its interior and $\mu(\cdot) = \widetilde{C}_q(K,\cdot)$. 
}

The dual Minkowski problem contains two special cases: when $q=0$, it becomes the Aleksandrov problem; when $q=n$, it becomes the logarithmic Minkowski problem.

In this paper, we will consider the dual Minkowski problem for the case $q<0$. A complete solution, including the existence and the uniqueness part, will be presented. Note that this works not only in the even case, but in the case when the given measure is arbitrary as well. In particular, the main theorems in this paper are:

\begin{customthm}{(Existence part of the dual Minkowski problem for negative indices)}
Suppose $q<0$ and $\mu$ is a non-zero finite Borel measure on $S^{n-1}$. There exists a convex body $K\subset \mathbb{R}^n$ that contains the origin in its interior, such that $\mu(\cdot)=\widetilde{C}_q(K,\cdot)$ if and only if $\mu$ is not concentrated in any closed hemisphere.
\end{customthm}
\begin{customthm}{(Uniqueness part of the dual Minkowski problem for negative indices)}
Suppose $q<0$ and $K,L$ are two convex bodies that contain the origin in their interiors. If $\widetilde{C}_q(K,\cdot) = \widetilde{C}_q(L,\cdot)$, then $K=L$.
\end{customthm}

Dual curvature measures and dual Minkowski problem are concepts belonging to the dual Brunn-Minkowski theory initiated by Lutwak (see Schneider \cite{schneider2014}). The theory started out by replacing support functions by radial functions, and mixed volumes by dual mixed volumes. The dual Brunn-Minkowski theory has been most effective in dealing with questions related to intersections, while the Brunn-Minkowski theory has been most helpful in answering questions related to projections. One of the major triumphs of the dual Brunn-Minkowski theory is tackling the famous Busemann-Petty problem, see Gardner \cite{MR1298719}, Gardner, Koldobsky \& Schlumprecht \cite{MR1689343}, Lutwak \cite{MR963487}, and Zhang \cite{MR1689339}. Although the duality demonstrated in the dual Brunn-Minkowski theory is only heuristic, it has provided numerous strikingly similar (formally), yet significant concepts and results, see, for example, Gardner \cite{MR2353261}, Gardner, Hug \& Weil \cite{MR3120744}, Haberl \cite{MR2397461}, Haberl \& Ludwig \cite{MR2250020}, and Zhang \cite{MR1443203}. Also see Gardner \cite{MR2251886} and Schneider \cite{schneider2014} for a detailed account. 

Recall that dual quermassintegrals are the means of the intersection areas of convex bodies (see \eqref{eq intro dual quermassintegral}). Define the normalized $(n-i)$-th dual quermassintegral $\bar{W}_{n-i}(K)$ of $K\in \mathcal{K}_o^n$ to be 
\begin{equation*}
\bar{W}_{n-i}(K)=\left(\frac{1}{\omega_n}\widetilde{W}_{n-i}(K)\right)^\frac{1}{i}.
\end{equation*}
Both dual quermassintegrals and normalized dual quermassintegrals can be naturally extended to real indices, see \eqref{eq dual quermassintegral}, \eqref{eq normalized dual quermassintegral 1}, and \eqref{eq normalized dual quermassintegral 2}. For each $q\in \mathbb{R}$, the $q$-th dual curvature measure, denoted by $\widetilde{C}_q(K,\cdot)$, of a convex body $K$ containing the origin in its interior, may be defined to be the unique Borel measure on $S^{n-1}$ such that
\begin{equation}
\label{eq intro variational formula}
\left.\frac{d}{dt}\log \bar{W}_{n-q}([K,f]_t)\right|_{t=0}=\frac{1}{\widetilde{W}_{n-q}(K)}\int_{S^{n-1}}f(v)d\widetilde{C}_q(K,v),
\end{equation}
holds for each continuous $f:S^{n-1}\rightarrow \mathbb{R}$. Here $[K,f]_t$ is the \emph{logarithmic Wulff shape} generated by $K$ and $f$, i.e.,
\begin{equation*}
[K,f]_t=\{x\in \mathbb{R}^n:x\cdot v\leq h_K(v)e^{tf(v)},\forall v\in S^{n-1}\}.
\end{equation*}
Of critical importance is the fact that dual curvature measures are valuations, i.e., $$\widetilde{C}_q(K,\cdot)+\widetilde{C}_q(L,\cdot)= \widetilde{C}_q(K\cup L,\cdot)+\widetilde{C}_q(K\cap L,\cdot),$$ for each $K,L\in \mathcal{K}_o^n$ such that $K\cup L\in \mathcal{K}_o^n$. See, e.g., Haberl \cite{MR2966660}, Haberl \& Ludwig \cite{MR2250020} ,Haberl \& Parapatits \cite{MR3194492,MR3176613}, Ludwig \cite{MR2159706,MR2772547}, Ludwig \& Reitzner \cite{MR2680490}, Schuster \cite{MR2435426,MR2668553}, Schuster \& Wannerer \cite{MR2846354} and the references therein for important valuations and their characterizations in the theory of convex bodies.

The current paper aims to give a complete solution, including existence and uniqueness, to the dual Minkowski problem for negative indices.

\section{Preliminaries}
\label{section preliminaries}
\subsection{Basics regarding convex bodies}
Books such as \cite{MR2251886} and \cite{schneider2014} often serve as good references for the theory of convex bodies.

We will mainly be working in $\mathbb{R}^n$ equipped with the usual Euclidean norm $|\cdot|$. The standard inner product will be written as $x\cdot y$ for vectors $x,y\in \mathbb{R}^n$. The standard $n$-dimensional unit ball will be denoted by $B_n$ and its volume by $\omega_n$. Write $S^{n-1}$ for the boundary of $B_n$ and recall that its surface area is $n\omega_n$. We will use $C(S^{n-1})$ to denote the space of continuous functions on $S^{n-1}$ with the usual max norm; i.e., $||f||=\max\{|f(u)|:u\in S^{n-1}\}$. We will also write $C^+(S^{n-1})$ for the set of positive continuous functions on $S^{n-1}$. For a given measure $\mu$, we will use $|\mu|$ for its total measure.

A subset $K$ of $\mathbb{R}^n$ is called a \emph{convex body} if it is a compact convex set with non-empty interior. The set of all convex bodies that contain the origin in the interior is denoted by $\mathcal{K}_o^n$. The boundary of $K$ will be denoted by $\partial K$. 

Associated to each convex body $K\in \mathcal{K}_o^n$ is the support function $h_K:S^{n-1}\rightarrow \mathbb{R}$ given by
\begin{equation*}
h_K(v)=\max \{v\cdot x: x\in K\},
\end{equation*}
for each $v\in S^{n-1}$. It is easy to see that $h_K$ is a continuous function and $h_K>0$. Hence, the support function $h_K$ is bounded away from $0$.

Another function that can be associated to a convex body $K\in \mathcal{K}_o^n$ is the radial function $\rho_K$. Define $\rho_K:S^{n-1}\rightarrow \mathbb{R}$ by 
\begin{equation*}
\rho_K(u)=\max\{\lambda>0:\lambda u\in K\},
\end{equation*}
for each $u\in S^{n-1}$. Again, it can be seen that $\rho_K$ is a continuous function and $\rho_K>0$. Hence, the radial function $\rho_K$ is bounded away from $0$.

The set $\mathcal{K}_o^n$ can be endowed with two metrics: the Hausdorff metric is such that the distance between $K,L\in \mathcal{K}_o^n$ is $||h_K-h_L||$; the radial metric is such that the distance between $K,L\in \mathcal{K}_o^n$ is $||\rho_K-\rho_L||$. Note that the two metrics are equivalent, i.e., if $K\in \mathcal{K}_o^n$ and $K_1,\cdots,K_n,\cdots\in \mathcal{K}_o^n$, then
\begin{equation*}
h_{K_i} \rightarrow h_K \text{ uniformly} \qquad\text{ if and only if } \qquad\rho_{K_i} \rightarrow \rho_K \text{ uniformly.}
\end{equation*}
Thus, we may write $K_i\rightarrow K$ without specifying which metric is in use.

For each $K\in \mathcal{K}_o^n$, we can define the polar body $K^*$ by
\begin{equation*}
K^*=\{x\in \mathbb{R}^n:x\cdot y\leq 1, \text{ for all }y\in K\}.
\end{equation*}
Note that $K^*\in \mathcal{K}_o^n$ and by its definition, we have
\begin{equation*}
\rho_K = 1/h_{K^*} \qquad\text{ and } \qquad h_K=1/\rho_{K^*}.
\end{equation*}
We will also be using the fact that if $K\in \mathcal{K}_o^n$ and $K_1,\cdots, K_n,\cdots\in \mathcal{K}_o^n$, then 
\begin{equation}
\label{eq polar convergence}
K_i\rightarrow K \qquad\text{ if and only if } \qquad K_i^*\rightarrow K^*.
\end{equation}

For each $f\in C^+(S^{n-1})$, define $[f] \in \mathcal{K}_o^n$ to be the Wulff shape generated by $f$, i.e., 
\begin{equation*}
[f]= \{x\in \mathbb{R}^n: x\cdot v \leq f(v) \text{ for all } v\in S^{n-1}\}.
\end{equation*}
Obviously, one has 
\begin{equation}
\label{radial of convex hull}
h_{[f]}\leq f,
\end{equation}
and 
\begin{equation}
\label{radial of convex hull of radial}
[ h_K] = K,
\end{equation} 
for each $K\in \mathcal{K}_o^n$. 
Suppose $K\in \mathcal{K}_o^n$ and $g\in C(S^{n-1})$. When $\delta>0$ is small enough, we can define $h_t\in C^+(S^{n-1})$ by
\begin{equation}
\label{logarithmic convex hull}
\log h_t(v)=\log h_K(v)+tg(v),
\end{equation}
for each $v\in S^{n-1}$ and $t\in (-\delta,\delta)$. We will usually write $[K,g,t]$ for the convex body $[h_t]$.

Let us assume $K\in \mathcal{K}_o^n$. The \emph{supporting hyperplane} $P(K,v)$ of $K$ for each $v\in S^{n-1}$ is given by
\begin{equation*}
P(K,v)=\{x\in \mathbb{R}^n:x\cdot v = h_K(v)\}.
\end{equation*}
At each boundary point $x\in \partial K$, a unit vector $v$ is said to be an \emph{outer unit normal} of $K$ at $x\in \partial K$ if $P(K,v)$ passes through $x$. 

For a subset $\omega\subset S^{n-1}$, the \emph{radial Gauss image}, $\vec{\alpha}_K(\omega)$, of $K$ at $\omega$, is the set of all outer unit normals of $K$ at points in $\{\rho_K(u)u:u\in \omega\}$. When $\omega=\{u\}$, we usually write $\vec{\alpha}_K(u)$ instead of $\vec{\alpha}_K(\{u\})$. Let $\omega_K\subset S^{n-1}$ be the set consisting of all $u\in S^{n-1}$ such that the set $\vec{\alpha}_K(u)$ contains more than one single element. It can be shown that the set $\omega_K$ is of spherical Lebesgue measure $0$ (see Theorem 2.2.5 in \cite{schneider2014}). The \emph{radial Gauss map}, $\alpha_K:S^{n-1}\setminus \omega_K\rightarrow S^{n-1}$ is the map that takes each $u\in S^{n-1}\setminus \omega_K$ to the unique element in $\vec{\alpha}_K(u)$. 

Similarly, for each $\eta\subset S^{n-1}$, the \emph{reverse radial Gauss image}, $\vec{\alpha}^*_K(\eta)$, of $K$ at $\eta$, is the set of all radial directions $u\in S^{n-1}$ such that the boundary point $\rho_K(u)u$ has at least one element in $\eta$ as its outer unit normal, i.e., $\vec{\alpha}^*_K(\eta)=\{u\in S^{n-1}:\vec{\alpha}_K(u)\cap\eta \neq \emptyset\}$. When $\eta = \{v\}$, we usually write $\vec{\alpha}_K^*(v)$ instead of $\vec{\alpha}_K^*(\{v\})$. Let $\eta_K\subset S^{n-1}$ be the set consisting of all $v\in S^{n-1}$ such that the set $\vec{\alpha}_K^*(v)$ contains more than one single element. The set $\eta_K$ is of spherical Lebesgue measure $0$ (see Theorem 2.2.11 in \cite{schneider2014}). The \emph{reverse radial Gauss map}, $\alpha_K^*:S^{n-1}\setminus \eta_K \rightarrow S^{n-1}$, is the map that takes each $v\in S^{n-1}\setminus \eta_K$ to the unique element in $\vec{\alpha}_K^*(v)$. 

A detailed description of the radial Gauss map and the reverse radial Gauss map, together with a list of their properties, can be found in Section 2.2 \cite{HLYZ}.

While central to the Brunn-Minkowski theory are quermassintegrals, the geometric functionals central to the dual Brunn-Minkowski theory are \emph{dual quermassintegrals}. Let $q\in \mathbb{R}$ and $K\in \mathcal{K}_o^n$, the $(n-q)$-th dual quermassintegral $\widetilde{W}_{n-q}(K)$ may be defined by 
\begin{equation}
\label{eq dual quermassintegral}
\widetilde{W}_{n-q}(K)= \frac{1}{n}\int_{S^{n-1}}\rho_K^q(u)du.
\end{equation}
For real $q\neq 0$, the \emph{normalized dual quermassintegral} $\bar{W}_{n-q}(K)$ is given by 
\begin{equation}
\label{eq normalized dual quermassintegral 1}
\bar{W}_{n-q}(K)=\left(\frac{1}{n\omega_n}\int_{S^{n-1}}\rho_K^q(u)du\right)^{\frac{1}{q}},
\end{equation}
and for $q=0$, by 
\begin{equation}
\label{eq normalized dual quermassintegral 2}
\bar{W}_n(K)= \exp\left(\frac{1}{n\omega_n}\int_{S^{n-1}}\log \rho_K(u)du\right).
\end{equation}
We will write 
\begin{equation*}
\widetilde{V}_q(K)= \widetilde{W}_{n-q}(K)\qquad \text{ and }\qquad \bar{V}_q(K) = \bar{W}_{n-q}(K).
\end{equation*}
The functionals $\widetilde{V}_q$ and $\bar{V}_q$ are called the \emph{$q$-th dual volume} and the \emph{normalized $q$-th dual volume}, respectively.

\subsection{Dual curvature measures and the dual Minkowski problem}
\label{sec dual curvature measures}
For quick later references, we gather here some facts about dual curvature measures and the dual Minkowski problem.

While curvature measures can be viewed as differentials of quermassintegrals, they can also be defined by considering \emph{local parallel sets} (see the construction in Chapter 4 \cite{schneider2014}). In the spirit of \emph{conceptual duality}, Huang-LYZ \cite{HLYZ} were able to discover dual curvature measures $\widetilde{C}_q(K,\cdot)$ for $q\in \mathbb{R}$ using what they call \emph{local dual parallel bodies}. See Section 3 in \cite{HLYZ} for a detailed construction of dual curvature measures. This new family of geometric measures can also be viewed as differentials of dual quermassintegrals (see \eqref{eq intro variational formula}) and thus warrant being called ``dual'' curvature measures. Dual curvature measures have the following integral representation. For each Borel set $\eta\subset S^{n-1}$ and $q\in \mathbb{R}$, the dual curvature measure $\widetilde{C}_q(K,\cdot)$ of $K\in\mathcal{K}_o^n$ can be defined by
\begin{equation}
\label{eq curvature measure integral representation}
\widetilde{C}_q(K,\eta)=\frac{1}{n}\int_{\vec{\alpha}_K^*(\eta)}\rho_K^q(u)du.
\end{equation}
The total measure of $\widetilde{C}_q(K,\cdot)$ is equal to the $(n-q)$-th dual quermassintegral, i.e.,
\begin{equation}
\widetilde{W}_{n-q}(K)=\widetilde{C}_q(K,S^{n-1}).
\end{equation}
It is not hard to see that $\widetilde{C}_q$ is homogeneous of degree $q$. That is 
\begin{equation*}
\widetilde{C}_q(\lambda K,\cdot)=\lambda^q\widetilde{C}_q(K,\cdot),
\end{equation*}
for each $\lambda>0$. From the integral representation \eqref{eq curvature measure integral representation} and Lemma 2.5 in \cite{HLYZ}, it is not hard to see that the $0$-th dual curvature measure is (up to a constant) equal to the Aleksandrov's integral curvature for the polar body and the $n$-th curvature measure is (up to a constant) equal to the cone volume measure. 

When $K$ has smooth boundary with everywhere positive Gauss curvature, the dual curvature measure $\widetilde{C}_q(K,\cdot)$ of $K$ is absolutely continuous with respect to the spherical Lebesgue measure and the density is given (in terms of the support function of $K$) by
\begin{equation*}
d\widetilde{C}_q(K,v)=\frac{1}{n}h_K(v)|\nabla h_K(v)|^{q-n}\det(h_{ij}(v)+h_K(v)\delta_{ij})dv,
\end{equation*}
where $(h_{ij})$ is the Hessian matrix of $h_K$ on $S^{n-1}$ with respect to an orthonormal basis. When $K$ is a polytope with outer unit normal $\{v_1,\cdots, v_m\}$, the dual curvature measure $\widetilde{C}_{q}(K,\cdot)$ is a discrete measure concentrated on $\{v_1,\cdots,v_m\}$ and is given by
\begin{equation*}
\widetilde{C}_q(K,\cdot)= \sum_{i=1}^m c_i\delta_{v_i},
\end{equation*}
where $\delta_{v_i}$ is the Dirac measure concentrated at $v_i$ and
\begin{equation*}
c_i = \frac{1}{n} \int_{\vec{\alpha}_K^*(v_i)}\rho_K^q(u)du.
\end{equation*}
See \cite{HLYZ} for details.

It is natural to look for requirements on a given measure so that it becomes the $q$-th dual curvature measure of a convex body $K$ for a given $q\in \mathbb{R}$. The characterization problem for dual curvature measures, called the dual Minkowski problem, was posed in \cite{HLYZ}:

\textbf{The dual Minkowski problem:}\textit{
Given a finite Borel measure $\mu$ on $S^{n-1}$ and $q\in \mathbb{R}$, find the necessary and sufficient condition(s) on $\mu$ so that it becomes the $q$-th dual curvature measure of a convex body $K\in \mathcal{K}_o^n$. 
}

When $q=0$, the dual Minkowski problem is the same as the Aleksandrov problem. Both the existence and the uniqueness of the solution to Aleksandrov problem were given by Aleksandrov \cite{MR0007625}. See also Oliker \cite{MR2332603} for an intriguingly new approach and its connection to optimal transport. When $q=n$, the dual Minkowski problem becomes the logarithmic Minkowski problem whose complete solution still remains open. When restricted to the even case, the existence part of the logarithmic Minkowski problem was established by B{\"o}r{\"o}czky-LYZ \cite{BLYZ}. For non-even cases, Stancu \cite{MR1901250,MR2019226} studied the problem in the planar case. Zhu \cite{MR3228445} gave a sufficient condition when the given measure is discrete, but not necessarily even. B\"{o}r\"{o}czky, Heged\H{u}s, \& Zhu \cite{Boroczky20062015} later found a sufficient condition for the discrete case, which includes \cite{BLYZ} (in the discrete case) and \cite{MR3228445} as special cases. See also Chou \& Wang \cite{MR2254308} for the case when the given measure has a positive density. The uniqueness part of the logarithmic Minkowski problem, in general, remains open (see B{\"o}r{\"o}czky-LYZ \cite{MR2964630} and Stancu \cite{MR1901250,MR2019226} for results in the planar case).

When $0<q<n$, the dual Minkowski problem was considered in Huang-LYZ \cite{HLYZ}. Restricting to the class of even measures and origin-symmetric convex bodies, they found a sufficient condition that would guarantee the existence of a solution to the dual Minkowski problem. 

\begin{thm}[\cite{HLYZ}]
Suppose $\mu$ is a non-zero finite even Borel measure on $S^{n-1}$ and $q\in (0,n]$. If the measure $\mu$ satisfies: 
\begin{enumerate}[(1)]
\item when $q\in [1,n]$,
\begin{equation}
\label{eq tmp 0000}
\frac{\mu(S^{n-1}\cap \xi_{n-i})}{|\mu|}<1-\frac{i(q-1)}{(n-1)q},
\end{equation}
for all $i<n$ and all $(n-i)$ dimensional subspaces $\xi_{n-i}\subset \mathbb{R}^n$;
\item when $q\in (0,1)$,
\begin{equation*}
\frac{\mu(S^{n-1}\cap \xi_{n-1})}{|\mu|}<1,
\end{equation*}
for all $(n-1)$ dimensional subspaces $\xi\subset\mathbb{R}^n$,
\end{enumerate}
then there exists an origin-symmetric convex body $K$ in $\mathbb{R}^n$ such that 
\begin{equation*}
\widetilde{C}_q(K,\cdot)=\mu(\cdot).
\end{equation*}
\end{thm}

In the following, the dual Minkowski problem when $q<0$ will be considered.

\section{The Optimization Problem}
\label{section optimization problem}
The first step towards solving various kinds of Minkowski problems using variational method is to properly convert the original problem to an optimization problem whose Euler-Lagrange equation would imply that the given measure is equal to the geometric measure (under investigation) of an optimizer. The associated optimization problem for the dual Minkowski problem was asked in \cite{HLYZ}. It was also established that for even measures, a solution to the optimization problem will lead to a solution to the dual Minkowski problem. Note that the proof works essentially in the same way even if the measure is not even. For the sake of completeness, we shall first describe the optimization problem and then give a short account of how a solution to the optimization problem would lead to a solution to the dual Minkowski problem.

Suppose $\mu$ is a non-zero finite Borel measure. Since we are dealing with the dual Minkowski problem for negative indices, we may restrict our attention to $0\neq q \in \mathbb{R}$. Define $\Phi:C^+(S^{n-1})\rightarrow \mathbb{R}$ by letting
\begin{equation}
\label{definition of Phi}
\Phi(h)=-\frac{1}{|\mu|}\int_{S^{n-1}}\log h(v)d\mu(v)+\log \bar{V}_q([h]),
\end{equation}
for every $h\in C^+(S^{n-1})$. Note that the functional $\Phi$ is homogeneous of degree $0$; i.e., 
\begin{equation}
\label{eq_homogeneity of Phi}
\Phi(ch)=\Phi(h),
\end{equation}
for all $c>0$.

When $q=n$, the functional $\Phi$ becomes
\begin{equation*}
\Phi(h)= - \frac{1}{|\mu|}\int_{S^{n-1}}\log h(v)d\mu(v)+\frac{1}{n}\log(\text{vol}([h])/\omega_n)
\end{equation*}
and is a key component in solving the even logarithmic Minkowski problem in B{\"o}r{\"o}czky-LYZ \cite{BLYZ}.

When $q=0$, the functional $\Phi$ becomes 
\begin{equation*}
\Phi(h) = -\frac{1}{|\mu|}\int_{S^{n-1}}\log h(v)d\mu (v)+\frac{1}{n\omega_n}\int_{S^{n-1}}\log \rho_{[h]}(u)du
\end{equation*}
and in a slightly different form was studied in Oliker \cite{MR2332603}.

\textbf{The optimization problem (I):}
\begin{equation*}
\sup\{\Phi(h):\widetilde{V}_q([h])=|\mu|,h\in C^+(S^{n-1})\}.
\end{equation*}

Note that for each $h\in C^+(S^{n-1})$, by \eqref{radial of convex hull} and \eqref{radial of convex hull of radial}, 
\begin{equation*}
\Phi(h)\leq \Phi(h_{[h]}) \text{ and } \widetilde{V}_q([h]) = \widetilde{V}_q([h_{[h]}]).
\end{equation*}
Thus, we may restrict our attention in the search of a maximizer to the set of all support functions. That is, $h_{Q_0}$ is a maximizer to the optimization problem (I) if and only if $Q_0$ is a maximizer to the following optimization problem:

\textbf{The optimization problem (II):}
\begin{equation*}
\sup\{\Phi_\mu(K):\widetilde{V}_q(K)=|\mu|,K\in \mathcal{K}_o^n\},
\end{equation*}
where $\Phi_\mu:\mathcal{K}_o^n\rightarrow \mathbb{R}$ is defined by letting 
\begin{equation*}
\Phi_\mu(K)=-\frac{1}{|\mu|}\int_{S^{n-1}}\log h_K(v)d\mu(v)+\log\bar{V}_q(K),
\end{equation*}
for each $K\in \mathcal{K}_o^n$. Note that on $\mathcal{K}_o^n$, the functional $\Phi_\mu$ is continuous with respect to the Hausdorff metric.

One potential difficulty in obtaining the Euler-Lagrange equation for the optimization problem (I) or (II) is that taking the differential of $\Phi$ or $\Phi_\mu$, more specifically the functional $\log \bar{V}_q([h])$, can be hard. The following variational formula was established in \cite{HLYZ} (see Theorem 4.5):
\begin{equation}
\label{variational formula}
\left.\frac{d}{dt} \log \bar{V}_q([K,g,t])\right|_{t=0}=\frac{1}{\widetilde{V}_q(K)}\int_{S^{n-1}}g(v)d\widetilde{C}_q(K,v),
\end{equation}
where $[K,g,t]$ is defined in \eqref{logarithmic convex hull} and $g$ is an arbitrary continuous function on $S^{n-1}$. With the help of \eqref{variational formula}, we may now show how a maximizer to the optimization (I), or equivalently (II), will lead to a solution to the dual Minkowski problem.

Suppose $Q_0$ is a maximizer to (II), or equivalently $h_{Q_0}$ is a maximizer to (I); i.e., $\widetilde{V}_q(Q_0)=|\mu|$ and 
\begin{equation*}
\Phi(h_{Q_0})=\sup\{\Phi(h):\widetilde{V}_q([h])=|\mu|,h\in C^+(S^{n-1})\}.
\end{equation*}
By \eqref{eq_homogeneity of Phi}, 
\begin{equation}
\label{eq_tmp_optimization problem 1}
\Phi(h_{Q_0})\geq \Phi(h),
\end{equation}
for each $h\in C^+(S^{n-1})$. Let $g:S^{n-1}\rightarrow \mathbb{R}$ be an arbitrary continuous function. For $\delta>0$ small enough and $t\in (-\delta,\delta)$, define $h_t:S^{n-1}\rightarrow \mathbb{R}$ by 
\begin{equation*}
h_t = h_{Q_0}e^{tg}.
\end{equation*}

By \eqref{eq_tmp_optimization problem 1}, \eqref{definition of Phi}, \eqref{variational formula}, and the fact that $\widetilde{V}_q(Q_0)=|\mu|$,
\begin{equation*}
\begin{aligned}
0 &= \left.\frac{d}{dt}\Phi(h_t)\right|_{t=0}\\
  &= \left.\frac{d}{dt}\left(-\frac{1}{|\mu|}\int_{S^{n-1}}\log h_{Q_0}(v)+tg(v)d\mu(v)+\log \bar{V}_q([h_t])\right)\right|_{t=0}\\
  &=-\frac{1}{|\mu|}\int_{S^{n-1}}g(v)d\mu(v)+\frac{1}{\widetilde{V}_q(Q_0)}\int_{S^{n-1}}g(v)d\widetilde{C}_q(Q_0,v)\\
  &=\frac{1}{|\mu|}\left(-\int_{S^{n-1}}g(v)d\mu(v)+\int_{S^{n-1}}g(v)d\widetilde{C}_q(Q_0,v)\right).
\end{aligned}
\end{equation*}
Since this holds for any arbitrary function $g\in C(S^{n-1})$, we have 
\begin{equation*}
\mu(\cdot) = \widetilde{C}_q(Q_0,\cdot).
\end{equation*}

Thus, we have
\begin{lem}
\label{lemma optimization problem}
Suppose $q<0$ and $\mu$ is a non-zero finite Borel measure. Assume $Q_0\in \mathcal{K}_o^n$. If $\widetilde{V}_q(Q_0)=|\mu|$ and 
\begin{equation*}
\Phi_\mu(Q_0)=\sup\{\Phi_\mu(K):\widetilde{V}_q(K)=|\mu|,K\in \mathcal{K}_o^n\},
\end{equation*}
then
\begin{equation*}
\mu(\cdot)=\widetilde{C}_q(Q_0,\cdot).
\end{equation*}
\end{lem}

Note that the above lemma works for other $q$'s as well. But since only the dual Minkowski problem for negative indices is considered here, we choose to state the lemma only for $q<0$.

\section{Solving the Optimization Problem}
This section is dedicated to showing that the optimization problem (II) has a maximizer when the given measure $\mu$ is not concentrated in any closed hemisphere. This, together with Lemma \ref{lemma optimization problem}, immediately implies that the dual Minkowski problem when $q<0$ has a solution.

The following lemma gives an upper bound for the polar of convex bodies that have fixed $q$-th dual volume.
\begin{lem}
\label{lemma boundedness}
Suppose $q<0$ and $c>0$. Assume $K\in \mathcal{K}_o^n$. If 
\begin{equation*}
\widetilde{V}_q(K)=\frac{1}{n}\int_{S^{n-1}}\rho_K^{q}(u)du = c,
\end{equation*}
then there exists $M=M(c)>0$ such that
\begin{equation*}
K^*\subset M B_n.
\end{equation*}
\end{lem}
\begin{proof}
First we note that by the rotational invariance of the spherical Lebesgue measure, the integral 
\begin{equation*}
\int_{S^{n-1}}(u\cdot v)_+^{-q}du
\end{equation*}
is independent of the choice of $v\in S^{n-1}$. Here $(u\cdot v)_+=\max\{u\cdot v,0\}$. Since the spherical Lebesgue measure is not concentrated in any closed hemisphere,
\begin{equation}
\label{value of an integral}
m_0 := \int_{S^{n-1}}(u\cdot v)_+^{-q}du >0.
\end{equation}

Let $v_0\in S^{n-1}$ be such that 
\begin{equation*}
\rho_{K^*}(v_0)= \max_{v\in S^{n-1}}\rho_{K^*}(v).
\end{equation*}
By definition of the support function and the fact that $K\in \mathcal{K}_o^n$, 
\begin{equation*}
h_{K^*}(u)\geq (u\cdot v_0)_+\rho_{K^*}(v_0).
\end{equation*}
This, the fact that $q<0$, and \eqref{value of an integral} imply
\begin{equation*}
\begin{aligned}
c = \widetilde{V}_q(K)&=\frac{1}{n}\int_{S^{n-1}}\rho_K^{q}(u)du\\
					&=\frac{1}{n}\int_{S^{n-1}}h_{K^*}^{-q}(u)du\\
				   &\geq \frac{1}{n}\int_{S^{n-1}}(u\cdot v_0)_+^{-q}\rho_{K^*}^{-q}(v_0)du\\
				   &=\frac{1}{n} m_0\rho_{K^*}^{-q}(v_0).
\end{aligned}
\end{equation*}
This implies that 
\begin{equation*}
\rho_{K^*}(v_0)\leq \left(\frac{nc}{m_0}\right)^{-\frac{1}{q}}.
\end{equation*}
By the choice of $v_0$, we may choose $M=\left(\frac{nc}{m_0}\right)^{-\frac{1}{q}}$ and thus
\begin{equation*}
K^*\subset MB_n.
\end{equation*}
\qed
\end{proof}

The next lemma will solve the optimization problem (II).
\begin{lem}
\label{existence of maximizer}
Suppose $q<0$ and $\mu$ is a non-zero finite Borel measure. If $\mu$ is not concentrated in any closed hemisphere, then there exists $Q_0\in \mathcal{K}_o^n$ with $\widetilde{V}_q(Q_0) = |\mu|$ and 
\begin{equation*}
\Phi_\mu(Q_0)=\sup\{\Phi_\mu(K):\widetilde{V}_q(K)=|\mu| \text{ and } K\in \mathcal{K}_o^n\}.
\end{equation*}
\end{lem}
\begin{proof}
Suppose $\{Q_i\}\subset \mathcal{K}_o^n$ is a maximizing sequence; i.e., $\widetilde{V}_q(Q_i)= |\mu|$ and
\begin{equation}
\label{Q_i is a maximizing sequence}
\lim_{i\rightarrow \infty}\Phi_\mu(Q_i)=\sup\{\Phi_\mu(K):\widetilde{V}_q(K)=|\mu| \text{ and }K\in \mathcal{K}_o^n\}.
\end{equation}
By Lemma \ref{lemma boundedness}, there exists $M=M(|\mu|)>0$ such that 
\begin{equation}
\label{bound for Q_i}
Q_i^*\subset MB_n.
\end{equation}
By Blaschke's selection theorem, we may assume (by taking subsequence) that there exists a compact convex set $K_0\subset \mathbb{R}^n$ such that 
\begin{equation*}
Q_i^*\rightarrow K_0.
\end{equation*} 

We note that if $o\in \text{int }K_0$, then we are done. Indeed, we may take $Q_0=K_0^*$. To see why this works, we may use the continuity of $\Phi_\mu$ and $\widetilde{V}_q$, and \eqref{eq polar convergence} to conclude that
\begin{equation*}
\widetilde{V}_q(Q_0)= \widetilde{V}_q(K_0^*)=\lim_{i\rightarrow \infty}\widetilde{V}_q(Q_i)=|\mu|,
\end{equation*}
and
\begin{equation*}
\Phi_\mu(Q_0)=\Phi_\mu(K_0^*)=\lim_{i\rightarrow \infty}\Phi_\mu(Q_i)=\sup \{\Phi_\mu(K):\widetilde{V}_q(K)=|\mu| \text{ and }K\in \mathcal{K}_o^n\}.
\end{equation*}

Let us now show that $o\in \text{int } K_0$. Assume otherwise, i.e., $o\in \partial K_0$. Hence there exists $u_0\in S^{n-1}$ such that $h_{K_0}(u_0)=0$. Since $Q_i^*$ converges to $K_0$ in Hausdorff metric, we have 
\begin{equation}
\label{support funtion goes to 0}
\lim_{i\rightarrow \infty} h_{Q_i^*}(u_0)= h_{K_0}(u_0)= 0.
\end{equation}
For $0<\delta<1$, define
\begin{equation*}
\omega_\delta(u_0) = \{v\in S^{n-1}:v\cdot u_0>\delta\}.
\end{equation*}
For each $v\in \omega_\delta(u_0)$,
\begin{equation*}
h_{Q_i^*}(u_0)\geq (v\cdot u_0)\rho_{Q_i^*}(v)> \delta \rho_{Q_i^*}(v),
\end{equation*}
which implies
\begin{equation*}
\rho_{Q_i^*}(v)< h_{Q_i^*}(u_0)/\delta.
\end{equation*}
This, together with \eqref{support funtion goes to 0}, shows that $\rho_{Q_i^*}$ converges to $0$ uniformly on $\omega_\delta(u_0)$.

By monotone convergence theorem and the fact that $\mu $ is a finite measure that is not concentrated in any closed hemisphere, we have 
\begin{equation*}
\lim_{\delta\rightarrow 0} \mu\left(\omega_\delta(u_0)\right)= \mu\left(\{v\in S^{n-1}:v\cdot u_0>0\}\right)>0.
\end{equation*}
This implies the existence of $\delta_0>0$ such that
\begin{equation}
\label{positive measure}
\mu\left(\omega_{\delta_0}(u_0)\right)>0.
\end{equation}
Hence, by \eqref{bound for Q_i}, \eqref{positive measure}, and the fact that $\rho_{Q_i^*}$ converges to $0$ uniformly on $\omega_{\delta_0}(u_0)$,
\begin{equation*}
\begin{aligned}
&\Phi_\mu(Q_i) \\=&-\frac{1}{|\mu|}\int_{\omega_{\delta_0}(u_0)}\log h_{Q_i}(v)d\mu(v)-\frac{1}{|\mu|}\int_{S^{n-1}\setminus \omega_{\delta_0}(u_0)}\log h_{Q_i}(v)d\mu(v)+\frac{1}{q}\log \frac{|\mu|}{\omega_n}\\
=&\frac{1}{|\mu|}\int_{\omega_{\delta_0}(u_0)}\log \rho_{Q_i^*}(v)d\mu(v)+\frac{1}{|\mu|}\int_{S^{n-1}\setminus \omega_{\delta_0}(u_0)}\log \rho_{Q_i^*}(v)d\mu(v)+\frac{1}{q}\log \frac{|\mu|}{\omega_n}\\
\leq & \frac{1}{|\mu|}\int_{\omega_{\delta_0}(u_0)}\log \rho_{Q_i^*}(v)d\mu(v)+\frac{1}{|\mu|}\mu(S^{n-1}\setminus \omega_{\delta_0}(u_0))\log M +\frac{1}{q}\log \frac{|\mu|}{\omega_n}\\
\rightarrow& -\infty,
\end{aligned}
\end{equation*}
as $i\rightarrow \infty$. This is clearly a contradiction to $\{Q_i\}$ being a maximizing sequence.\qed
\end{proof}

Lemmas \ref{lemma optimization problem} and \ref{existence of maximizer} immediately give the following theorem.
\begin{thm}
Suppose $q<0$ and $\mu$ is a finite non-zero Borel measure. There exists $K\in \mathcal{K}_o^n$ such that $\widetilde{C}_q(K,\cdot)=\mu(\cdot)$ if and only if $\mu$ is not concentrated in any closed hemisphere.
\end{thm}
\begin{proof}
The only if part is obvious, while the if part follows from Lemmas \ref{lemma optimization problem} and \ref{existence of maximizer}.\qed
\end{proof}

\section{Uniqueness}
For the dual Minskowski problem when $q<0$, not only does a solution exist, the uniqueness of the solution can be established as well. The primary goal of this section is to establish this fact. 

The following lemma is needed.

\begin{lem}
\label{lemma uniqueness}
Suppose $Q_1,Q_2\in \mathcal{K}_o^n$. If the following sets
\begin{equation*}
\eta_1=\{v\in S^{n-1}:h_{Q_1}(v)>h_{Q_2}(v)\},
\end{equation*}
\begin{equation*}
\eta_2=\{v\in S^{n-1}:h_{Q_1}(v)<h_{Q_2}(v)\},
\end{equation*}
\begin{equation*}
\eta_0=\{v\in S^{n-1}:h_{Q_1}(v)=h_{Q_2}(v)\},
\end{equation*}
are non-empty, then the following statements are true:
\begin{enumerate}[(a)]
\item \label{(1)}If $u\in \vec{\alpha}_{Q_1}^*(\eta_1)$, then $\rho_{Q_1}(u)>\rho_{Q_2}(u)$;
\item \label{(2)}If $u\in \vec{\alpha}_{Q_2}^*(\eta_2\cup \eta_0)$, then $\rho_{Q_2}(u)\geq \rho_{Q_1}(u)$;
\item \label{(3)}$\vec{\alpha}_{Q_1}^*(\eta_1)\subset \vec{\alpha}_{Q_2}^*(\eta_1)$;
\item \label{(4)}$\mathcal{H}^{n-1}(\vec{\alpha}_{Q_2}^*(\eta_1))>0$ and $\mathcal{H}^{n-1}(\vec{\alpha}_{Q_1}^*(\eta_2))>0$. 
\end{enumerate}
\end{lem}
\begin{proof}

\begin{enumerate}[(a)]
\item We prove by contradiction. Assume that $\rho_{Q_1}(u)\leq \rho_{Q_2}(u)$. Since $u\in \vec{\alpha}_{Q_1}^*(\eta_1)$, there exists $v_0\in \eta_1$ such that $u\cdot v_0>0$ and $\rho_{Q_1}(u)u\cdot v_0=h_{Q_1}(v_0)$. Hence,
\begin{equation*}
h_{Q_1}(v_0)=\rho_{Q_1}(u)u\cdot v_0\leq \rho_{Q_2}(u)u\cdot v_0\leq h_{Q_2}(v_0).
\end{equation*}
This is a contradiction to the fact that $v_0\in \eta_1$.
\item We again prove by contradiction. Assume $\rho_{Q_1}(u)>\rho_{Q_2}(u)$. Since $u\in \vec{\alpha}_{Q_2}^*(\eta_2\cup \eta_0)$, there exists $v_0\in \eta_2\cup\eta_0$ such that $u\cdot v_0>0$ and $\rho_{Q_2}(u)u\cdot v_0 = h_{Q_2}(v_0)$. Hence,
\begin{equation*}
h_{Q_1}(v_0)\geq \rho_{Q_1}(u)u\cdot v_0>\rho_{Q_2}(u)u\cdot v_0=h_{Q_2}(v_0).
\end{equation*}
This is a contradiction to the fact that $v_0\in \eta_2\cup \eta_0$.
\item Suppose $u\in S^{n-1}$ is such that $u\in \vec{\alpha}_{Q_1}^*(\eta_1)$ but $u\notin \vec{\alpha}_{Q_2}^*(\eta_1)$. Then $u\in \vec{\alpha}_{Q_1}^*(\eta_1)\cap \vec{\alpha}_{Q_2}^*(\eta_2\cup \eta_0)$. Then \eqref{(1)} and \eqref{(2)} provide a contradiction.
\item By symmetry, we only need to show $\mathcal{H}^{n-1}(\vec{\alpha}_{Q_2}^*(\eta_1))>0$. Suppose $$\mathcal{H}^{n-1}(\vec{\alpha}_{Q_2}^*(\eta_1))=0.$$
Then by \eqref{(2)},
\begin{equation}
\label{eq local 02}
\rho_{Q_2}(u)\geq \rho_{Q_1}(u),
\end{equation}
for $\mathcal{H}^{n-1}$ almost all $u\in S^{n-1}$. By continuity of radial functions, \eqref{eq local 02} is valid for all $u\in S^{n-1}$. This implies $Q_1\subset Q_2$, which is a contradiction to $\eta_1$ being non-empty.\qed
\end{enumerate}
\end{proof}

The following theorem establishes the uniqueness for the solution to the dual Minkowski problem for negative $q's$.

\begin{thm}
Assume $q<0$ and $K,L\in \mathcal{K}_o^n$. If $\widetilde{C}_q(K,\cdot)=\widetilde{C}_q(L,\cdot)$, then $K=L$.
\end{thm}
\begin{proof}
By homogeneity of $\widetilde{C}_q$, it suffices to show that $K$ is a dilate of $L$. Assume not. Then there exists $\lambda>0$ and $K'=\lambda K$ such that 
\begin{equation*}
\begin{aligned}
\eta'&=\{v\in S^{n-1}:h_{K'}(v)>h_L(v)\},\\
\eta &=\{v\in S^{n-1}:h_{K'}(v)<h_L(v)\},\\
\eta_0 &=\{v\in S^{n-1}:h_{K'}(v)=h_L(v)\},
\end{aligned}
\end{equation*}
are non-empty.

Lemma \ref{lemma uniqueness}\eqref{(4)}, together with the definition of $\widetilde{C}_q$, shows that $\widetilde{C}_q(L,\eta')>0$. This, in turn, implies that 
\begin{equation}
\label{eq local 1000}
\widetilde{C}_q(K,\eta')>0.
\end{equation}
This implies that
\begin{equation}
\label{eq local 1001}
\mathcal{H}^{n-1}(\vec{\alpha}_{K'}^*(\eta'))=\mathcal{H}^{n-1}(\vec{\alpha}_K^*(\eta'))>0.
\end{equation}

By Lemma \ref{lemma uniqueness}\eqref{(3)}, Lemma \ref{lemma uniqueness}\eqref{(1)}, the fact that $q$ is negative, \eqref{eq local 1001}, and the homogeneity of $\widetilde{C}_q$, we have
\begin{equation}
\label{eq local 03}
\begin{aligned}
\widetilde{C}_q(K,\eta')&=\widetilde{C}_q(L,\eta')\\
&=\frac{1}{n}\int_{\vec{\alpha}_L^*(\eta')}\rho_L^q(u)du\\
&\geq \frac{1}{n}\int_{\vec{\alpha}_{K'}^*(\eta')}\rho_L^q(u)du\\
&>\frac{1}{n}\int_{\vec{\alpha}_{K'}^*(\eta')}\rho_{K'}^q(u)du\\
&=\widetilde{C}_q(K',\eta')\\
&=\lambda^q\widetilde{C}_q(K,\eta').
\end{aligned}
\end{equation}

Hence \eqref{eq local 1000} and \eqref{eq local 03} imply that 
\begin{equation}
\label{eq local 05}
\lambda^q<1.
\end{equation}

Similarly, Lemma \ref{lemma uniqueness}\eqref{(4)}, together with the definition of $\widetilde{C}_q$, shows that $$\widetilde{C}_q(K',\eta)>0.$$
This, in turn, implies that 
\begin{equation}
\label{eq local 1002}
\widetilde{C}_q(K,\eta)>0.
\end{equation}
Thus $\widetilde{C}_q(L,\eta)>0$, which implies that
\begin{equation}
\label{eq local 1003}
\mathcal{H}^{n-1}(\vec{\alpha}_L^*(\eta))>0.
\end{equation}

By Lemma \ref{lemma uniqueness}\eqref{(1)}, the fact that $q$ is negative, \eqref{eq local 1003}, Lemma \ref{lemma uniqueness}\eqref{(3)}, and the homogeneity of $\widetilde{C}_q$,
\begin{equation}
\label{eq local 04}
\begin{aligned}
\widetilde{C}_q(K,\eta)&=\widetilde{C}_q(L,\eta)\\
&=\frac{1}{n}\int_{\vec{\alpha}_L^*(\eta)}\rho_L^q(u)du\\
&<\frac{1}{n}\int_{\vec{\alpha}_L^*(\eta)}\rho_{K'}^{q}(u)du\\
&\leq \frac{1}{n}\int_{\vec{\alpha}_{K'}^*(\eta)}\rho_{K'}^q(u)du\\
&=\widetilde{C}_q(K',\eta)\\
&=\lambda^q\widetilde{C}_q(K,\eta).
\end{aligned}
\end{equation}

Hence \eqref{eq local 1002} and \eqref{eq local 04} imply that
\begin{equation}
\label{eq local 06}
\lambda^q>1.
\end{equation}

There is a contradiction between \eqref{eq local 05} and \eqref{eq local 06}.\qed
\end{proof}


\end{document}